\documentclass[amstex,12pt]{article}%
\usepackage{graphicx}
\usepackage{amsfonts}
\usepackage{amssymb}
\usepackage{amsthm}
\usepackage{amsmath}%
\newtheorem{thm}{Theorem}
\newtheorem{lm}[thm]{Lemma}
\newtheorem{cor}[thm]{Corollary}
\newtheorem{rmk}{Remark}
\newtheorem*{dfn}{Definition}
\newtheorem*{sumthm}{Summarizing Theorem}
\input{Macros.tex}
\def\df{\def}
\df\P{{\mathcal P}}
\df\Q{{\mathcal Q}}
\df\D{{\mathcal D}}
\df\X{{\mathcal X}}
\df\G{{\Gamma}}
\df\g{{\gamma}}
\df\l{{\lambda}}
\df\L{{\Lambda}}
\df\k{{\kappa}}
\df\o{{\omega}}
\df\O{{\Omega}}
\df\r{{\rho}}
\df\s{{\sigma}}
\df\d{{\delta}}
\df\e{{\varepsilon}}
\df\a{{\alpha}}
\df\b{{\beta}}
\df\t{{\theta}}
\df\R{{\mathbb{R}}}
\df\N{{\mathbb{N}}}
\begin{document}

\author{Jo\"{e}l Rouyer
\and Costin V\^{\i}lcu}
\title{Simple closed geodesics \\on most Alexandrov surfaces}
\maketitle

\begin{abstract}
We study the existence of simple closed geodesics on most (in the sense of
Baire category) Alexandrov surfaces with curvature bounded below, compact and
without boundary. We show that it depends on both the curvature bound and the
topology of the surfaces.

\end{abstract}

\bigskip

{\small Math. Subj. Classification (2010): 53C45, 53C22}



\section{Introduction}

The existence of closed geodesics is of certain interest in the geometry of
Riemannian surfaces, and was studied in many articles. We mention here only a
very few facts, related to our topic. In this paper, whenever we consider
several geodesics they are geometrically distinct.

In the late nineteenth century, J. Hadamard \cite{Hadamard} showed that every
non-trivial homotopy class of closed curves on a closed Riemannian manifold
$M$ contains geodesics.

It is a famous result of L. A. Lusternik and L. G. Schnirelman that for every
Riemannian metric on the $2$-sphere there exist at least three simple closed
geodesics (and sometimes exactly three, \eg for ellipsoids with distinct axes)
\cite{k}. This was completed by a combined result of J. Franks and V. Bangert
\cite{f}, \cite{ba}, stating that for every metric on such a surface there
exist infinitely many closed geodesics.

On the other hand, for a given upper bound on the length, the number of closed
geodesics is usually finite.

M. Mirzakhani \cite{Mirzakhani} showed that the number $n_{X}(L)$, of simple
closed geodesics of length $\leq L$ on a hyperbolic Riemannian surface $X$ of
genus $g$, is asymptotic to $c_{X} L^{6g-6}$ as $L\to\infty$, where $c_{X}$ is
a constant depending on $X$.

G. Contreras \cite{Contreras} proved that for every closed manifold $M$ of
dimension at least two, there is an open and dense subset of the space of
$\mathcal{C}^{\infty}$ Riemannian metrics on $M$, any metric on which
satisfies $\lim_{L\rightarrow\infty}\frac{\log p(L)}{L}>0$, where $p(L)$ is
the number of closed geodesics of length $\leq L$.

Recall that Baire categories were previously employed in the study of
geodesics in the framework of Riemannian geometry. Improving previous results
of several authors, H. Rademacher \cite{Rademacher} proved that a
$\mathcal{C}^{r}$ typical metric on a compact simply connected manifold
carries infinitely many (not necessarily simple) closed geodesics ($2\leq
r\leq\infty$).

In this paper we consider the Baire space $\mathcal{A}(\kappa)$ of Alexandrov
surfaces (definitions below), in which smooth Riemannian surfaces form a set
of first category, even though dense. In this space, we study the existence of
simple closed geodesics on a typical surface, and show that it depends on both
the curvature bound and the topology of the surface.

\bigskip

Formally, we denote by $\mathcal{A}(\kappa)$ the set of all compact Alexandrov
surfaces with curvature bounded below by $\kappa$, without boundary. We refer
to \cite{BGP} or \cite{Shiohama92} for the precise definition and basic facts
about such spaces.

It is known that these surfaces are $2$-dimensional topological manifolds.
Closed Riemannian surfaces with Gauss curvature at least $\kappa$ and $\kappa
$-polyhedra (see $\S $\ref{Preliminaries} for the definition) are important
examples of such surfaces.

It is also known that, endowed with topology induced by the Gromov-Hausdorff
distance, $\mathcal{A}(\kappa)$ is a Baire space \cite{IRV2}. In any Baire
space, one says that \textit{most elements} or a \textit{typical element}
enjoys a property $P$ if the set of those elements which do not satisfy $P$ it
is of first category.

Let $\mathcal{A}\left(  \kappa,\chi\right)  $ denote the set of those surface
in $\mathcal{A}(\kappa)$ whose Euler-Poincar\'{e} characteristic is $\chi$.
The connected components of $\mathcal{A}\left(  \kappa\right)  $ are the sets
of those surfaces of a given topological type \cite{RV2}. Therefore,
$\mathcal{A}\left(  \kappa,\chi\right)  $ (if non-empty) is a connected
component of $\mathcal{A}\left(  \kappa\right)  $ if $\chi$ is positive or
odd, and is the union of two components otherwise.

\bigskip

The space of all convex surfaces in ${\mathbb{R}}^{3}$ is naturally endowed
with the Pompeiu-Hausdorff metric. By celebrated results of Alexandrov (for
existence, see \cite[p. 362]{al}) and Pogorelov (for rigidity, see \cite[p.
167]{Pog}), each surface $A\in\mathcal{A}(0,2)$ can be realized as a convex
surface in ${\mathbb{R}}^{3}$, unique up to an isometry of the ambient space.
Therefore, the intrinsic geometry of convex surfaces is a particular case of
the geometry of Alexandrov surfaces.

P. Gruber proved that most convex surfaces have no simple closed geode\-sics
\cite{grub1}, and later improved this result by dropping the simpleness
assumption \cite{grub}. His result strongly contrasts the mentioned result of
L. A. Lusternik and L. G. Schnirelman. Nevertheless, on any convex surface
there exist three simple closed quasi-geodesics \cite{p} (see for example
\cite[p. 373]{al} for the definition).

Adapted to our framework, P. Gruber's result states that \textsl{most
Alexandrov surfaces in $\mathcal{A}\left(  0,2\right)  $ have no (simple)
closed geodesics}. In this paper we investigate the typical existence -- or
non-existence -- of simple closed geodesics for the other values of $\kappa$
and $\chi$.

Notice that it suffices to study the curvature bounds $\kappa\in\{-1,0,1\}$,
because there is a natural homothety from $\mathcal{A}\left(  \kappa\right)  $
to $\mathcal{A}\left(  1\right)  $ if $\kappa>0$, and to $\mathcal{A}\left(
-1\right)  $ if $\kappa<0$. Also notice that $\mathcal{A}\left(
\kappa^{\prime}\right)  $ is nowhere dense in $\mathcal{A}\left(
\kappa\right)  $ for $\kappa^{\prime}>\kappa$, so a typical element in
$\mathcal{A}\left(  \kappa^{\prime}\right)  $ is not typical in $\mathcal{A}%
\left(  \kappa\right)  $.

\bigskip

Since the total curvature of a surface of $\mathcal{A}\left(  \kappa,0\right)
$ vanishes, the space $\mathcal{A}(0,0)$ contains only flat tori and flat
Klein bottles (see \cite[Lemma 4 ]{RV2}). It follows that each $A\in
\mathcal{A}\left(  0,0\right)  $ is union of simple closed geodesics.

In Section \ref{neg} we prove that most surfaces in $\mathcal{A}(-1)$ admit
infinitely many, non-intersecting, simple closed geodesics, and in Section
\ref{SRP2} we prove that most surfaces in $\mathcal{A}(\kappa,1)$ admit
infinitely many simple closed geodesics, all of bounded length. This contrasts
the mentioned result of M. Mirzakhani.

In Section \ref{last} we treat the remaining case -- $\mathcal{A}\left(
1,2\right)  $ -- and prove that a typical surface there has no simple closed geodesic.

\bigskip

Many properties of most convex surfaces have been investigated (see for
example the surveys \cite{gw} and \cite{z-b}), but only a few of them have
been hitherto generalized to Alexandrov surfaces (see \cite{A-Z},
\cite{IRV2}). In particular, most surfaces in $\mathcal{A}(\kappa)$ if
$\kappa\neq0$, and most surfaces in $\mathcal{A}(0)\setminus\mathcal{A}\left(
0,0\right)  $, are not Riemannian manifolds of class $\mathcal{C}^{2}$.


\section{Preliminaries}

\label{Preliminaries}

Let $H$ and $K$ be compact subsets of a metric space $Z$; we denote by
$d_{H}^{Z}\left(  H,K\right)  $ the usual Pompeiu-Hausdorff distance between
them. We shall omit the superscript $Z$ whenever no confusion is possible.

If $X$ and $Y$ are compact metric spaces, we denote by $d_{GH}\left(
X,Y\right)  $ the \emph{Gromov-Hausdorff distance} between $X$ and $Y$. For
its definition and basic properties, we refer to \cite{GPL} or \cite{BBI}.
Recall that we have $d_{GH}\left(  H,K\right)  \leq d_{H}^{Z}\left(
H,K\right)  $ for any compact subsets $H$, $K$ of a given metric space $Z$;
moreover, we have the following lemma.

\begin{lm}
\cit{JR7}\label{SEL} Let $\left\{  X_{n}\right\}  _{n\in\mathbb{N}}$ be a
sequence of compact metric spaces converging to $X$ with respect to the
Gromov-Hausdorff metric, and let $\left\{  \varepsilon_{n}\right\}
_{n\in\mathbb{N}}$ be a sequence of positive numbers. Then there exist a
compact metric space $Z$, an isometric embedding $\varphi:X\rightarrow Z$ and,
for each positive integer $n$, an isometric embedding $\varphi_{n}%
:X_{n}\rightarrow Z$, such that
\[
d_{H}^{Z}\left(  \varphi_{n}\left(  X_{n}\right)  ,\varphi\left(  Y\right)
\right)  <d_{GH}\left(  X_{n},X\right)  +\varepsilon_{n}.
\]

\end{lm}

A more sophisticated fact is the famous Perel'man's theorem of stability. The
reader will find a complete proof in \cite{Kapo1}, or in the original
manuscript \cite{Per1}; in our case ($2$-dimensional spaces without boundary)
the proof admits large simplifications. In order to give its statement, recall
the definition of \emph{distortion}. If $f:X\rightarrow Y$ is a map between
metric spaces then
\[
\mathrm{dis}\left(  f\right)  =\sup_{x,x^{\prime}\in X}\left\vert d\left(
x,x^{\prime}\right)  -d\left(  f(x),f(x^{\prime})\right)  \right\vert \text{.}%
\]

\begin{lm}
[Perel'man's stability theorem]\label{LP} Let $A_{n}$, $A\in\mathcal{A}%
(\kappa)$ and suppose that there exist functions $f_{n}:A\rightarrow A_{n}$
such that $\mathrm{dis}\left(  f_{n}\right)  \rightarrow0$. Then, for $n$
large enough, there exists homeomorphisms $h_{n}:A\rightarrow A_{n}$ such that
$\sup_{x\in A}d\left(  f_{n}\left(  x\right)  ,h_{n}\left(  x\right)  \right)
\rightarrow0$.
\end{lm}

Consider two surfaces $S$ and $S^{\prime}$ with boundaries $\partial S$ and
$\partial S^{\prime}$; assume there exist arcs $I\subset\partial S$ and
$I^{\prime}\subset\partial S^{\prime}$ having the same length. By
\textit{gluing $S$ to $S^{\prime}$ along $I$} we mean identifying the points
$x\in I$ and $\iota(x)\in I^{\prime}$, where $\iota:I\rightarrow I^{\prime}$
is a length preserving map between $I$ and $I^{\prime}$.

\begin{lm}
[Alexandrov's gluing theorem]\label{gluing}\cit{Pog} Let $S$ be a closed
topological surface obtained by gluing finitely many geodesic polygons cut out
from surfaces in $\mathcal{A}(\kappa)$, in such a way that the sum of the
angles glued together at each point is at most $2\pi$. Then, endowed with the
induced metric, $S$ belongs to $\mathcal{A}(\kappa)$.
\end{lm}

Let $\mathbb{M}_{\kappa}$ denote the simply-connected and complete Riemannian
surface of constant curvature $\kappa$.

A $\kappa$\emph{-polyhedron} is an Alexandrov surface obtained by gluing
finitely many geodesic polygons from $\mathbb{M}_{\kappa}$. Let $\mathcal{P}%
\left(  \kappa\right)  $ denote the set of all $\kappa$-polyhedra.

A formal proof for the next result can be found, for example, in \cite{IRV2}.

\begin{lm}
\label{Approximation} The subset of $\kappa$-polyhedra, and the subset of
closed Riemannian surfaces with Gauss curvature at least $\kappa$, are both
dense in $\mathcal{A}(\kappa)$.
\end{lm}

The length of a curve $\gamma$ will be denoted by $\mathcal{\ell}\left(
\gamma\right)  $.

\begin{lm}
\label{LA} Let $X$ be a compact metric space, and for each $n\in\mathbb{N}$
let $\gamma_{n}:\left[  0,1\right]  \rightarrow X$ be a rectifiable arc
parametrized proportionally to the arc-length. Assume that the sequence
$\left\{  \mathcal{\ell}\left(  \gamma_{n}\right)  \right\}  _{n}$ is bounded.
Then one can extract from it a subsequence converging uniformly to a
rectifiable arc $\gamma:\left[  0,1\right]  \rightarrow X$. Moreover,
$\mathcal{\ell}\left(  \gamma\right)  \leq\liminf\mathcal{\ell}\left(
\gamma_{n}\right)  $ and $\gamma_{n}\left(  \left[  0,1\right]  \right)  $
converges to $\gamma\left(  \left[  0,1\right]  \right)  $ for the
Pompeiu-Hausdorff metric.
\end{lm}

\begin{proof}
The choice of the parameter and the fact that $\left\{  \mathcal{\ell}\left(
\gamma_{n}\right)  \right\}  _{n}$ is bounded imply that $\gamma_{n}$ are
equi-continuous, hence the first statement follows from Ascoli's theorem. The
second statement is nothing but the semi-continuity of length (see for example
\cite[2.3.4.iv]{BBI}). The third statement is an obvious consequence of the
first one.
\end{proof}

If $P$ is a subset of a metric space $Z$ and $\rho$ a positive number, we
denote by $N_{\rho}\left(  P\right)  $ the $\rho$-neighbourhood of $P$ in $Z$,
namely
\[
N_{\rho}\left(  P\right)  =\left\{  x\in Z|\exists y\in P~d^{Z}\left(
x,y\right)  \leq\rho\right\}  \text{.}%
\]

We end this section with a notion of stability for simple closed geodesics,
which is essential in our proofs.

\begin{dfn}
Let $A\in\mathcal{A}(\kappa)$. A simple closed geodesic $G$ of $A$ is said to
be \emph{stable} if for any isometric embedding $\phi:A\rightarrow Z$ in any
metric space $Z$, and for any positive number $\delta$, there exists $\eta>0$
such that for any $A^{\prime}\in\mathcal{A}(\kappa)$ included in $Z$, if
$d_{H}^{Z}\left(  \phi\left(  A\right)  ,A^{\prime}\right)  \leq\eta$ then
there exists a simple closed geodesic $G^{\prime}$ in $A^{\prime}$ such that
$d_{H}^{Z}\left(  \phi\left(  G\right)  ,G^{\prime}\right)  \leq\delta$.
\end{dfn}


\section{A curvature argument}

\label{neg}

We recall first the \emph{Poincar\'{e}'s disc model} of $\mathbb{M}_{-1}$. It
consists of the standard open disk
\[
\mathbb{P}=\left\{  \left(  x,y\right)  |x^{2}+y^{2}<1\right\}
\]
endowed with the distance
\[
d_{\mathbb{P}}\left(  u,v\right)  =\mathrm{arccosh}\left(  1+\mathfrak{p}%
(u,v)\right)  \text{,}%
\]
where
\begin{equation}
\mathfrak{p}(u,v)=\frac{2\left\Vert u-v\right\Vert ^{2}}{\left(  1-\left\Vert
u\right\Vert ^{2}\right)  \left(  1-\left\Vert v\right\Vert ^{2}\right)
}\text{,} \label{frac}%
\end{equation}
and $\left\Vert ~ \right\Vert $ is the standard Euclidean norm. In this model,
the geodesics are exactly the circular arcs normal to the disk boundary.

\begin{lm}
\label{L1} Let $Q=Q\left(  \lambda,\varepsilon\right)  $ ($\lambda>0$,
$\varepsilon>0$) be the geodesic quadrilateral of $\mathbb{P}$ whose vertices
are $\left(  \pm a,\pm b\right)  $, where $a$, $b$ are chosen such that the
distance (in $\mathbb{P}$) between the midpoint of the upper side $U$ of $Q$
(from $\left(  a,b\right)  $ to $\left(  -a,b\right)  $) and the midpoint of
the lower side $L$ of $Q$ (from $\left(  a,-b\right)  $ to $\left(
-a,-b\right)  $) is $\lambda$, and the distance between the midpoints of the
other two sides of $Q$ is $\varepsilon$.

i) The unique shortest path $\gamma_{0}$ in $Q$ from $L$ to $U$ is a segment
of the $y$-axis.

ii) There exists a positive number $\alpha=\alpha\left(  \lambda
,\varepsilon\right)  $ such that any path $\gamma$ from $L$ to $U$
intersecting either the left or the right side of $Q$ satisfies $\mathcal{\ell
}\left(  \gamma\right)  \geq\mathcal{\lambda}+\alpha$.
\end{lm}

\begin{proof}
\textit{(i)} $Q$ is convex in $\mathbb{P}$, whence $\gamma_{0}$ is a geodesic
segment from $l\in L$ to $u\in U$. In order to maximize the denominator of
$\mathfrak{p}\left(  l,u\right)  $ in Formula (\ref{frac}), we have to chose
$l$ and $u$ on the $y$-axis. This condition also minimizes the numerator,
whence the conclusion.

\textit{(ii)} Assume the conclusion fails. So there exists a sequence
$\left\{  \gamma_{n}\right\}  _{n}$ of curves from $l_{n}\in L$ to $u_{n}\in
U$ via a point $r_{n}$ on (say) the right side $R$, such that $\mathcal{\ell
}\left(  \gamma_{n}\right)  \rightarrow\mathcal{\ell}\left(  \gamma
_{0}\right)  $. Let $m$ be the minimal value of the function $f:L\times
U\times R\rightarrow\mathbb{R}$ given by $\left(  l,u,r\right)  \mapsto
d_{\mathbb{P}}\left(  l,r\right)  +d_{\mathbb{P}}\left(  r,u\right)  $. By
\textit{(i)}, $m>\mathcal{\ell}\left(  \gamma_{0}\right)  $. On the other
hand, $\mathcal{\ell}\left(  \gamma_{n}\right)  \geq f\left(  l_{n}%
,r_{n},u_{n}\right)  \geq m$, whence $\mathcal{\ell}\left(  \gamma_{0}\right)
\geq m$ and we get a contradiction.
\end{proof}

We shall denote by $C\left(  \lambda,\varepsilon\right)  $ the manifold with
boundary obtained from the quadrilateral $Q\left(  \lambda,\varepsilon\right)
$ in Lemma \ref{L1} by gluing $L$ onto $U$, right onto right, and left onto
left. The segment which was the $y$-axis in $\mathbb{P}$ becomes after gluing
a simple closed geodesic. We call it the \emph{soul} of $C\left(
\lambda,\varepsilon\right)  $.

\begin{lm}
\label{stable} If $A\in\mathcal{A}(-1)$ contains a region $C$ isometric to
$C(\lambda,\varepsilon)$ for some $\lambda,\varepsilon>0$, then the soul of
$C$ is a stable simple closed geodesic.
\end{lm}

\begin{figure}[ptb]
\begin{center}
\includegraphics[
natheight=2.909100in,
natwidth=3.878800in,
height=1.7737in,
width=2.3557in
]{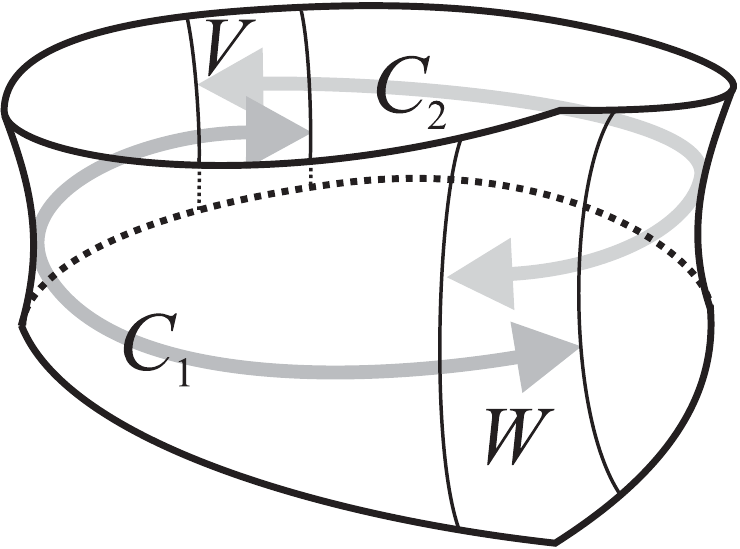}
\end{center}
\caption{Definition of $C_{1}$, $C_{2}$, $V$ and $W$ in the proof of Lemma
\ref{stable}.}%
\label{F1}%
\end{figure}

\begin{proof}
Let $G$ be the soul of $C=C\left(  \varepsilon,l\right)  $, with $C\subset A$.
Let $\phi:A\rightarrow Z$ be an isometric embedding of $A$ in some metric
space $Z$ and put $B=\phi\left(  A\right)  $. Choose $\delta>0$. Assume that
the result does not hold, hence there exists a sequence $\left\{
B_{n}\right\}  _{n}$ of Alexandrov surfaces isometrically embedded in $Z$ such
that $\nu_{n}\overset{\mathrm{def}}{=}d_{H}^{Z}\left(  B,B_{n}\right)  $ tends
to $0$, and $B_{n}$ has no simple closed geodesic $G^{\prime}$ with $d_{H}%
^{Z}\left(  G,G^{\prime}\right)  \leq\delta$. Define functions $f_{n}%
:B\rightarrow B_{n}$ (not necessary continuous) such that $d\left(
x,f_{n}\left(  x\right)  \right)  \leq\nu_{n}$; this is possible, by the
definition of the Pompeiu-Hausdorff distance.

By Lemma \ref{LP}, there exists a sequence of positive numbers $o_{n}$
convergent to $0$ such that, for large $n$, a homeomorphism $h_{n}%
:B\rightarrow B_{n}$ exists and satisfies $d\left(  h_{n}\left(  x\right)
,f_{n}\left(  x\right)  \right)  \leq o_{n}$. Hence $d_{n}\overset
{\mathrm{def}}{=}\mathrm{dis}\left(  h_{n}\right)  \leq\nu_{n}+2o_{n}%
\underset{_{n\rightarrow\infty}}{\longrightarrow}0$, and for all $x\in Z$ we
have $d^{Z}\left(  h_{n}\left(  x\right)  ,x\right)  \leq\nu_{n}%
+o_{n}\underset{_{n\rightarrow\infty}}{\longrightarrow}0$.

Let $\varepsilon^{\prime}$ be small enough to ensure that $C^{\prime}%
\overset{\mathrm{def}}{=}C\left(  \varepsilon^{\prime},l\right)  $ is included
in $N_{\delta/2}\left(  G\right)  $. For $n$ large enough, $h_{n}\left(
C^{\prime}\right)  \subset N_{\delta}\left(  G\right)  $.

Define two closed subset $C^{1}$, $C^{2}$ of $C^{\prime}$, delimited by
geodesics normal to $G$, such that $C^{\prime}=C_{1}\cup C_{2}$ and $C_{1}\cap
C_{2}$ is homeomorphic to the union of two closed ball, say $V$ and $W$ (see
Figure \ref{F1}). Let $\mathcal{K}$ (resp. $\mathcal{K}_{n}$) be the set of
those closed curves $\mathbb{R}/\mathbb{Z\rightarrow}B$ (resp. $\mathbb{R}%
/\mathbb{Z\rightarrow}B_{n}$), parametrized proportionally to the arc-length,
of length less than $2\mathcal{\ell}\left(  G\right)  $, and union of two arcs
from $v\in V$ (resp. $v_{n}\in h_{n}\left(  V\right)  $) to $w\in W$ (resp.
$w_{n}\in h_{n}\left(  W\right)  $), one of them lying in $C^{1}$ (resp.
$h_{n}\left(  C^{1}\right)  $) and the other in $C^{2}$ (resp. $h_{n}\left(
C^{2}\right)  $). By Lemma \ref{LA}, $\mathcal{K}_{n}$ is compact and there
exists a shortest curve $S_{n}:\mathbb{R}/\mathbb{Z\rightarrow}h_{n}\left(
C^{\prime}\right)  $ in $\mathcal{K}_{n}$. It is clear that, for $n$ large
enough, $d_{H}^{Z}\left(  \operatorname{Im}S_{n},G\right)  \leq\delta$. By our
assumption, $S_{n}$ is not a geodesic, and therefore intersects the boundary
of $h_{n}\left(  C^{\prime}\right)  $.

Assume (by passing to a subsequence, if necessary) that $S_{n}$ converges to
some closed curve $S\in\mathcal{K}$; then $S$ touches $\partial C^{\prime}$
and is not contractible in $C$. It follows (by Lemma \ref{L1}) that
$\mathcal{\mathcal{\ell}}\left(  S\right)  \geq\mathcal{\ell}\left(  G\right)
+\alpha\left(  \lambda,\varepsilon^{\prime}\right)  $, and (by Lemma \ref{LA})
that $\mathcal{\ell}\left(  S_{n}\right)  \geq\mathcal{\ell}\left(  G\right)
+\alpha\left(  \lambda,\varepsilon^{\prime}\right)  /2$ for $n$ large enough.

Let $v$ be the midpoint of $G\cap V$ and $w$ be the midpoint of $G\cap W$. Let
$G^{i}$ ($i=1$, $2$) be the part of $G$ delimited by $u$ and $v$ which is
contained in $C^{i}$. Take points $x_{0}=u$, $x_{1}$, \ldots, $x_{N^{1}}=v$ on
$G^{1}$ such that
\[
\max_{i}d\left(  x_{i},x_{i+1}\right)  \leq\frac{1}{2}d\left(  G,\partial
C^{1}\right)  \text{.}%
\]
Let $G_{n}^{1}\subset B_{n}$ be the union of segments from $h_{n}\left(
x_{i-1}\right)  $ to $h_{n}\left(  x_{i}\right)  $ $(1\leq i\leq N^{1})$; for
large $n$, $G_{n}^{1}\subset h_{n}\left(  C^{1}\right)  $. Moreover,
\[
\mathcal{\ell}\left(  G_{n}^{1}\right)  =\sum_{i=1}^{N}d\left(  h_{n}\left(
x_{i-1}\right)  ,h_{n}\left(  x_{i}\right)  \right)  \leq\sum_{i=1}%
^{N}d\left(  x_{i-1},x_{i}\right)  +N^{1}d_{n}\leq\mathcal{\ell}\left(
G^{1}\right)  +N^{1}d_{n}\text{.}%
\]
Similarly, one can construct $G_{n}^{2}\subset h_{n}\left(  C^{2}\right)  $.
The length of $G_{n}\overset{\mathrm{def}}{=}G_{n}^{1}\cup G_{n}^{2}$ is at
most $\mathcal{\ell}\left(  G\right)  +\left(  N^{1}+N^{2}\right)  d_{n}$. On
the other hand, $G_{n}\in\mathcal{K}_{n}$, whence $\mathcal{\ell}\left(
G_{n}\right)  \geq\mathcal{\ell}\left(  S_{n}\right)  \geq\mathcal{\ell
}\left(  G\right)  +\alpha\left(  \lambda,\varepsilon^{\prime}\right)  /2$,
and we get a contradiction.
\end{proof}

\begin{thm}
\label{TNC} Most surfaces in $\mathcal{A}(-1)$ have infinitely many,
non-intersecting, simple closed geodesics of bounded length.
\end{thm}

\begin{proof}
Let $\mathcal{G}_{p}\subset\mathcal{A}\left(  -1\right)  $ be the set of all
Alexandrov surfaces which admit at least $p$ non-intersecting simple closed
geodesics, and let $\mathcal{S}_{p}\subset\mathcal{G}_{p}$ be the set of all
Alexandrov surfaces which admit at least $p$ non-intersecting stable simple
closed geodesics.

We claim that $\mathcal{S}_{p}\subset\mathrm{int}\mathcal{G}_{p}$. Choose
$A\in\mathcal{S}_{p}$; we have to prove that for any sequence $A_{n}%
\in\mathcal{A}\left(  -1\right)  $ converging to $A$, $A_{n}$ belongs to
$\mathcal{G}_{p}$ for large $n$. By Lemma \ref{SEL}, we can assume that the
surfaces $A_{n}$, $A$ are all included in the same metric space $Z$. Let
$G_{1}$, \ldots, $G_{p}$ be $p$ non-intersecting and stable simple closed
geodesics of $A$. Put
\[
\delta=\frac{1}{3}\min_{1\leq i<j<p}\min_{\left(  x,y\right)  \in G_{i}\times
G_{j}}d\left(  x,y\right)  \text{.}%
\]
For $n$ large enough, there exists on $A_{n}$ a simple closed geodesic
$G_{i}^{n}$, lying in $N_{\delta}\left(  G_{i}\right)  $ ($i=1,...,p$), and
the geodesics $G_{1}^{n}$, $G_{2}^{n}$, ..., $G_{p}^{n}$ are non-intersecting
by the choice of $\delta$. This proves the claim.

We now claim that $\mathcal{S}_{p}$ is dense in $\mathcal{A}\left(  -1\right)
$. By Lemma \ref{Approximation}, it suffices to approximate every Riemannian
surface $R$ with surfaces in $\mathcal{S}_{p}$. $R$ admits at least one simple
closed geodesic $G$. A small neighbourhood of $G$ is homeomorphic to either a
cylinder or a M\"{o}bius strip.

Assume first that we are in the former case. Cutting $R$ along $G$ yields a
manifold $R^{\prime}$ whose boundary $\partial R^{\prime}$ consists of two
topological circles. For small $\varepsilon>0$, one can chose $\lambda$ such
that the boundary of $C\left(  \lambda,\varepsilon\right)  $ is isometric to
the boundary of $R^{\prime}$. Hence we can glue $p$ copies of $C_{1}$, \ldots,
$C_{p}$ of $C\left(  \lambda,\varepsilon\right)  $ between the two circles of
$\partial R^{\prime}$: one circle of $\partial R^{\prime}$ is glued to the
left side of $C_{1}$, the right side of $C_{i}$ ($1\leq i<p$) is glued on the
left side of $C_{i+1}$, and the right side of $C_{p}$ is glued on the other
circle of $\partial R^{\prime}$. By Lemmas \ref{gluing} and \ref{stable}, the
obtained surface belongs to $\mathcal{S}_{p}$, and for small $\varepsilon$ it
is close to $R$.

Assume now that a neighbourhood of $G$ is a M\"{o}bius strip, hence $\partial
R^{\prime}$ consists of one topological circle of length $2\mathcal{\ell
}\left(  G\right)  $. For small $\varepsilon$, we can choose $\lambda$ such
that each boundary component of $C\left(  \lambda,\varepsilon\right)  $ has
length $2\mathcal{\ell}\left(  G\right)  $. Glue successively $p$ copies of
$C_{1}$, \ldots, $C_{p}$ of $C\left(  \lambda,\varepsilon\right)  $ onto
$\partial R^{\prime}$: the left side of $C_{1}$ on $\partial R^{\prime}$ and
the right side of $C_{i}$ on the left side of $C_{i+1}$ ($1\leq i<p$). The
obtained surface still has a boundary, namely the right side of $C_{p}$. Glue
it on itself by identifying pairs of \textquotedblleft opposite
points\textquotedblright\ (\ie, points which are separating the boundary in
two arcs of length $\mathcal{\ell}\left(  G\right)  $). The obtained surface
belongs to $\mathcal{S}_{p}$ (by Lemmas \ref{gluing} and \ref{stable}) and is
closed to $R$. This proves the second claim.

It follows that $\mathrm{int}\left(  \mathcal{G}_{p}\right)  $ is open and
dense in $\mathcal{A}\left(  -1\right)  $, and
\begin{align*}
\mathcal{G}  &  =\set(:A\in\mathcal{A}\left(  -1\right)  |%
\begin{array}
[c]{l}%
A\text{~has infinitely many simple closed }\\
\text{geodesics pairwise non-intersecting}%
\end{array}
:)\\
&  \supset\bigcap_{p\in\mathbb{N}}\mathrm{int}\left(  \mathcal{G}_{p}\right)
\end{align*}
is residual in $\mathcal{A}\left(  -1\right)  $.

It is obvious from the above argument that the lengths of geodesics are bounded.
\end{proof}


\section{A topological argument}

\label{SRP2}

In the previous section we have proven the existence of simple closed
geodesics using a topology-free argument. In this section we shall use a
topology-based argument, which essentially does not depend on the curvature
bound. The case of $\mathcal{A}\left(  -1,1\right)  $ is covered by both
Section \ref{neg} and Section \ref{SRP2}.

The proof of the following easy lemma is left to the reader.

\begin{figure}[ptb]
\begin{center}
\includegraphics[scale=.5]{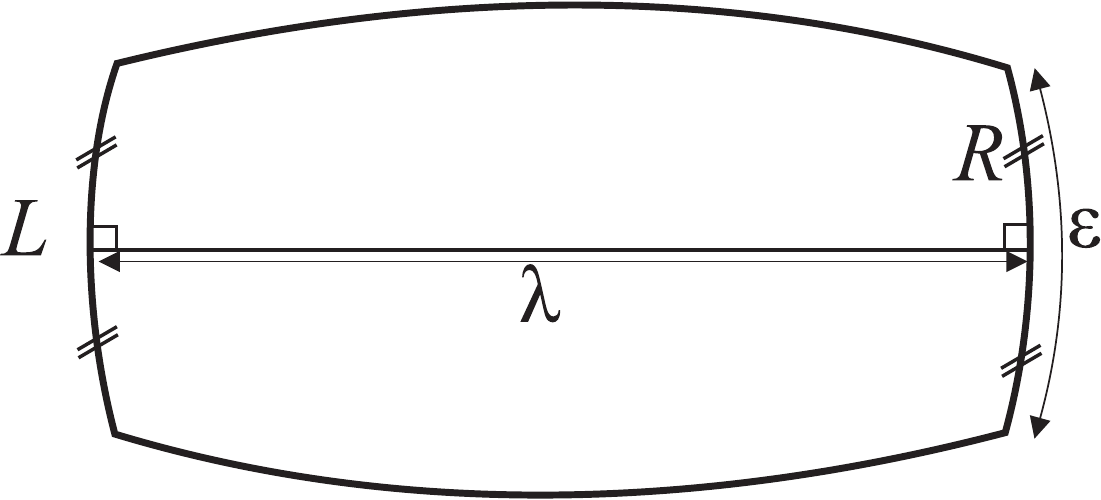}
\end{center}
\caption{Definition of $Q_{\kappa}(\lambda,\varepsilon)$ in Lemma \ref{LQK}.}%
\label{F4}%
\end{figure}

\begin{lm}
\label{LQK} Let $Q_{\kappa}=Q_{\kappa}\left(  \lambda,\varepsilon\right)  $ be
a geodesic quadrilateral in $\mathbb{M}_{\kappa}$ defined as in Figure
\ref{F4} ($\kappa=0$,$\pm1$), let $L$ be its left side and $R$ be its right
side. If $\kappa=1$ assume, moreover, that $\lambda<\pi$. Denote by $s$ the
symmetry with respect of its center.

The shortest curve from $x\in L$ to $s\left(  x\right)  $ is the segment
between the midpoints of $L$ and $R$. Moreover, there exists a positive number
$\beta=\beta\left(  \lambda,\varepsilon\right)  $ such that any curve from $x$
to $s\left(  x\right)  $ which touches either the upper or the lower side of
$Q_{\kappa}$ has a length of at least $\lambda+\beta$.
\end{lm}

Let $M_{\kappa}\left(  \lambda,\varepsilon\right)  $ be the compact M\"{o}bius
strip obtained from the quadrilateral $Q_{\kappa}\left(  \lambda
,\varepsilon\right)  $ in Lemma \ref{LQK} by gluing the two $\varepsilon$ long
sides. The segment joining the midpoints of the $\varepsilon$ long sides
becomes a simple closed geodesic in $M_{\kappa}\left(  \lambda,\varepsilon
\right)  $; call it the \emph{soul} of $M_{\kappa}\left(  \lambda
,\varepsilon\right)  $.

\begin{lm}
\label{stable_M} If $A\in\mathcal{A}\left(  \kappa\right)  $ contains a subset
$M$ isometric to some $M_{\kappa}\left(  \lambda,\varepsilon\right)  $ then
its soul is stable.
\end{lm}

\begin{proof}
By Lemma \ref{LQK}, there exist $\beta=\beta\left(  \lambda,\varepsilon
\right)  $ such that each non-contractible curve $\gamma\subset M$
intersecting $\partial M$ is longer that $\lambda+\beta$. From now on, the
proof is the same as the proof of Lemma \ref{stable}.
\end{proof}

\begin{cor}
\label{cor} Let $A\in\mathcal{P}\left(  \kappa\right)  $ be homeomorphic to
$\mathbb{RP}^{2}$ ($\kappa\in\{-1, 0, 1\}$), and let $G$ be a non-contractible
simple closed geodesic in $A$. If $\kappa=1$, assume moreover that
$\mathcal{\ell}\left(  G\right)  <\pi$. Then $G$ is stable.
\end{cor}

A\emph{ polyhedral disk} $D$ is a $2$-dimensional disk obtained by gluing a
finite collection of geodesic triangles of $\mathbb{M}_{\kappa}$, in such a
way that the sum of the angles glued together at each point is at most $2\pi$.
By definition, an \emph{angle} of $\partial D$ is a point whose space of
directions has a length distinct from $\pi$. This length will be called the
measure of the angle.

\begin{lm}
\label{LAD} Any polyhedral disk $D$ different from a half-sphere and whose
boundary has no angles can be approximated (with respect to the
Gromov-Haus\-dorff distance) by polyhedral disks whose boundary has two angles
of measure less than $\pi$, separating it in two equally long curves.
\end{lm}

\begin{proof}
We claim that $D$ has at least one vertex. If $\kappa\leq0$, this follows from
the Gauss-Bonnet Formula. If $\kappa=1$ and $D$ had no vertices, then gluing
two copies of it along its boundary would provide a simply connected
$1$-polyhedron without vertices. Such a polyhedron must be the standard
sphere, in contradiction with the fact that $D$ is not a half-sphere. Hence
$D$ contains at least one vertex $v$, say of singular curvature $\omega\left(
v\right)  $.

Choose two points $p,p^{\prime}\in\partial D$ separating $\partial D$ into two
arcs of equal length. Let $\sigma$ be a segment emanating from $p$ and normal
to $\partial D$, and let $q$ be a point of $\sigma$ close to $p$. Let $\gamma$
be a segment between $q$ and $v$; $\gamma\cap\partial D=\emptyset$, because
$q,v\not \in \partial D$ and $D$ is convex. Let $w$ be a point close to $v$
such that $\measuredangle qvw=\frac{2\pi-\omega\left(  v\right)  }{2}$; it
exists, because $D$ is polyhedral. Such $w$ is joined to $q$ by precisely two
segments, say $\gamma_{1}$, $\gamma_{2}$. Cut out from $D$ the digon they are
bounding and glue $\gamma_{1}$ onto $\gamma_{2}$. On the obtained disk
$D^{\prime}$, $q$ is a vertex of small singular curvature $\omega\left(
q\right)  $.

Consider a quadrilateral $abcb^{\prime}$ in $\mathbb{M}_{\kappa}$ such that
$d\left(  a,b\right)  =d\left(  a,b^{\prime}\right)  =d\left(  p,q\right)  $,
$\measuredangle abc=\measuredangle ab^{\prime}c=\pi/2$ and $\measuredangle
bab^{\prime}\leq\omega\left(  q\right)  $. Note that, if $d\left(  p,q\right)
<\pi/2$, we have $\measuredangle bcb^{\prime}<\pi$.

Cut $D^{\prime}$ along the arc $\sigma^{\prime}$ of $\sigma$ from $p$ to $q$
and glue $abcb^{\prime}$, $a$ at $q$ and the sides $ab$, $ab^{\prime}$ along
the two images of $\sigma^{\prime}$. The resulting disk boundary has one angle
at $c$.

Do the same construction starting at the point $p^{\prime}$, to obtain the
desired approximation of $D$.
\end{proof}

An \emph{almost-geodesic} $G$ on a $\kappa$-polyhedron is a polygonal line
admitting at each of its points $x$ (except its endpoints, if any) two tangent
directions, dividing the space of directions at point $x$ in two curves, at
least one of which has length $\pi$.

The proof of the next simple result is left to the reader.

\begin{lm}
\label{LLSCG} Let $P\in\mathcal{A}\left(  \kappa,2\right)  $ be a $\kappa
$-polyhedron whose vertices have singular curvature less than $\pi$. Let
$\{\Gamma_{n}\}$ be a sequence of geodesics on $P$ converging to
$\Gamma\subset P$ with respect to the Pompeiu-Hausdorff distance. Then
$\Gamma$ is an almost-geodesic.
\end{lm}

We denote by $S_{\alpha}\in\mathcal{P}\left(  1\right)  $ the orientable
surface obtained by gluing the two sides of a digon in $\mathbb{M}_{1}$ of
angle $2 \pi- \alpha$.

\begin{lm}
\label{LDPi} If $P\in\mathcal{A}\left(  1,2\right)  $ is a $1$-polyhedron then
$\mathrm{diam}(P)\leq\pi$, with equality if and only if $P=S_{\alpha}$ for
some $\alpha\in\left[  0,2\pi\right[  $.
\end{lm}

\begin{proof}
The inequality $\mathrm{diam}(A)\leq\pi$ is well-known for any $A\in
\mathcal{A}\left(  1\right)  $ (see \cite[Theorem 3.6]{BGP}), and all surfaces
$S_{\alpha}$ have diameter $\pi$.

Let $u$, $v\in P$ such that $\mathrm{diam}\left(  P\right)  =\pi=d\left(
u,v\right)  $. Consider a triangle $uvx$ in $P$ and let $\tilde{u}\tilde
{v}\tilde{x}$ be a comparison triangle on the sphere $\mathbb{M}_{1}$. We have
$\measuredangle uxv\geq\measuredangle\tilde{u}\tilde{x}\tilde{v}=\pi$. It
follows that the union of the segments $ux$ and $xv$ is a geodesic on $P$,
hence $x$ is not a vertex. The conclusion follows from the fact that the only
$1$-polyhedra with at most $2$ vertices are the surfaces $S_{\alpha}$
\cite{Troyanov}.
\end{proof}

\bigskip

The following lemma is a variant of a result of V. A. Toponogov, see for
example \cite{t} or \cite[p. 297]{k}.

\begin{lm}
\label{LHS} Let $G$ be a simple closed almost-geodesic of length $2\pi$ on the
$1$-polyhedron $P\in\mathcal{A}\left(  1,2\right)  $. If the boundary of one
of the two half-surfaces bounded by $G$ has no angles then this half-surface
is isometric to a half-sphere.
\end{lm}

\begin{proof}
Let $C$ be the half-surface of $P$ whose boundary has no angles. If $x$ is a
point of $G$, we denote by $x^{\prime}$ the point on $G$ such that
$G\setminus\left\{  x,x^{\prime}\right\}  $ consists two equally long arcs. By
the use of a (non trivial) comparison argument, it follows that $G$ is the
union of two segments between $x$ and $x^{\prime}\in G$, see the proof of
Theorem 3.4.10 in \cite[p. 297]{k}.

Now choose $p\in G$ and glue $C$ on itself by identifying points $x\in G$ and
$y\in G$ such that $d\left(  x,p\right)  =d\left(  y,p\right)  $. Since $G$ is
the union of two segments, the diameter of the obtained surface is $\pi$,
hence this surface is $S_{\pi}$ (by Lemma \ref{LDPi}) and $C$ is the standard half-sphere.
\end{proof}

The following lemma follows directly from Lemma \ref{LHS}.

\begin{lm}
\label{length} The length of a simple closed geodesic $G$ on a $1$-polyhedron
$A\in\mathcal{A}(1,1)$ satisfies $\mathcal{\ell}\left(  G\right)  \leq\pi$,
with equality if and only if $A$ is the projective space with constant
curvature $1$.
\end{lm}

\begin{thm}
\label{generic_scg} Most surfaces in $\mathcal{A}(\kappa,1)$ have infinitely
many simple closed geodesics of bounded length.
\end{thm}

\begin{proof}
Denote by $\mathcal{S}_{m}$ the set of those surfaces in $\mathcal{A}%
(\kappa,1)$ which admit at least $m$ stable simple closed geodesics. We only
need to prove that $\mathcal{S}_{m}$ is dense; afterwards the proof proceeds
in the same way as the proof of Theorem \ref{TNC}.

Let $P_{0}\in A\left(  \kappa,1\right)  $ be the real projective plane of
constant curvature. Choose $A\in\mathcal{A}(\kappa,1)\setminus\left\{
P_{0}\right\}  $ and approximate $A$ by a polyhedron $P\neq P_{0}$. The
shortest non-contractible closed curve on $P$ is a geodesic $G$. Note that, by
Lemma \ref{length}, if $\kappa=1$ then $\mathcal{\ell}\left(  G\right)  <\pi$.
Cutting $P$ along $G$ provides a polyhedral disk $D$. By Lemma \ref{LAD}, $D$
can be approximated by polyhedral disks $D^{\prime}$ whose boundary has two
angles of measure $\pi-\alpha_{0}$, for small positive $\alpha_{0}$,
separating it in two curves of equal length $L\approx\mathcal{\ell}\left(
G\right)  /2$.

\begin{figure}[ptb]
\begin{center}
\includegraphics[
scale=.4,
]{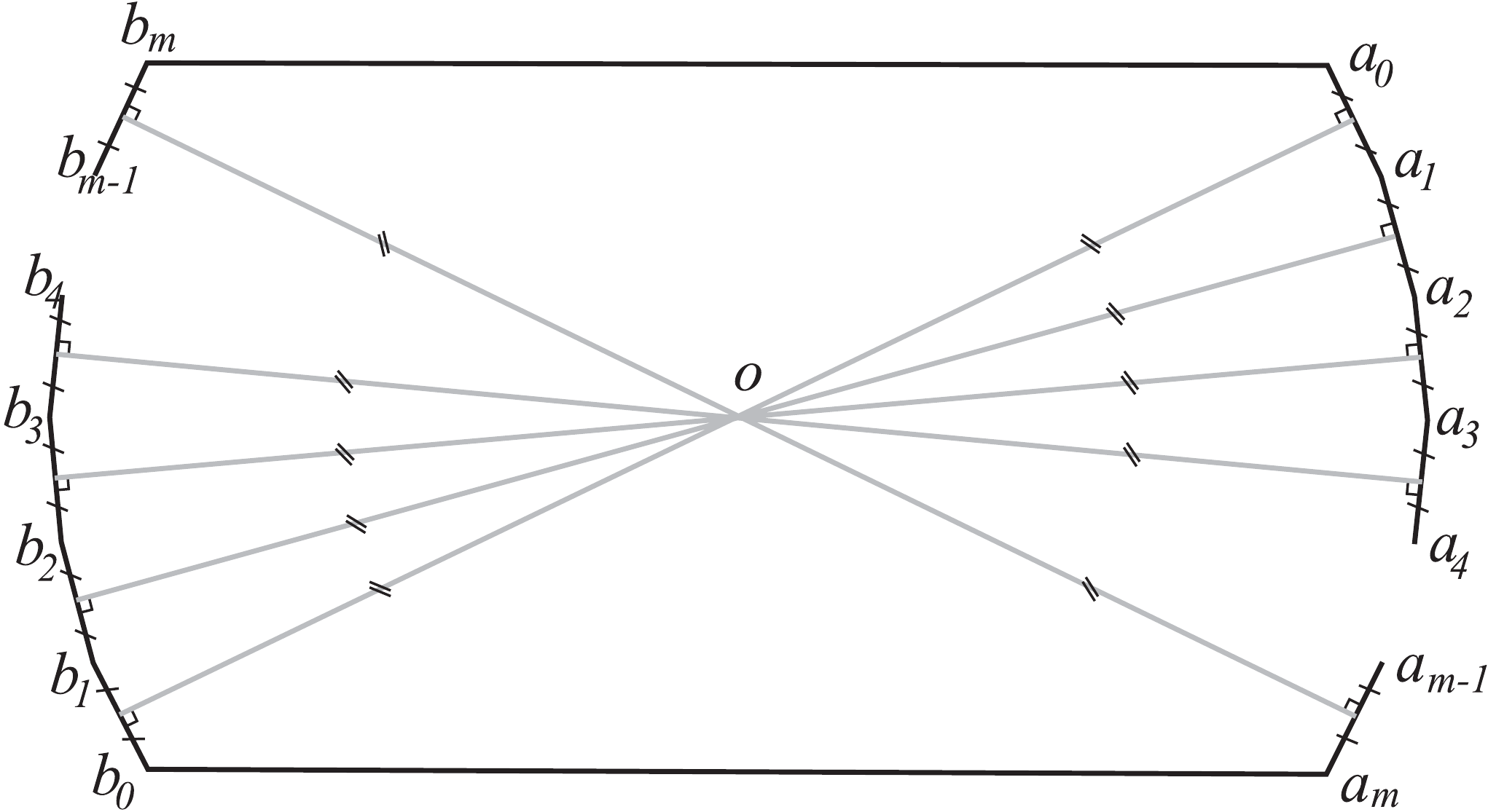}
\end{center}
\caption{Definition of $\Pi_{\kappa}\left(  m,\varepsilon\right)  $ in the
proof of Theorem \ref{generic_scg} (in the case $\kappa=0$).}%
\label{F2}%
\end{figure}

Consider in $\mathbb{M}_{\kappa}$ the $\left(  2m+2\right)  $-gon $\Pi
_{\kappa}\left(  m,\lambda,\varepsilon\right)  =a_{0}a_{1}\ldots a_{m}%
a_{0}b_{1}\ldots b_{m}$ defined as in Figure \ref{F2}, where $\varepsilon
=d\left(  a_{i},a_{i-1}\right)  =d\left(  b_{i},b_{i-1}\right)  $ ($i=1$,
\ldots, $m$) and $\lambda$ is the distance between mid-points of opposite
edges (\ie, the length of a gray line in Figure \ref{F2}). Glue the side
$a_{i}a_{i-1}$ onto the side $b_{i}b_{i-1}$ ($i=1$, \ldots, $m$), to obtain a
surface $\Lambda_{\kappa}(\lambda,\varepsilon)$ homeomorphic to a M\"{o}bius
strip. Its boundary has two angles of measure $\pi+\alpha$ (with $\alpha>0$
and tending to $0$ when $\varepsilon$ tends to $0$) separating it into two
equally long arcs. One can adjusts the parameters $\lambda$ and $\varepsilon$
such that the boundary length of $\Lambda_{\kappa}(\lambda,\varepsilon)$ is
exactly $2L$, and such that $\alpha\leq\alpha_{0}$. So we can glue this
$\Lambda_{\kappa}(\lambda,\varepsilon)$ to the boundary of $D^{\prime}$. The
resulting surface (which still belongs to $\mathcal{A}(1,1)\cap\mathcal{P}%
\left(  1\right)  $) approaches $P$ when $\varepsilon\rightarrow0$. It is
clear that this surface admits at least $m$ non-contractible simple closed
geodesics, corresponding to the gray lines in Figure \ref{F2}. These geodesics
are stable by Corollary \ref{cor}, proving the density of $\mathcal{S}_{m}$ in
$\mathcal{A}(\kappa,1)$.

It is clear from the above argument that the lengths of geodesics are bounded.
The proof is complete.
\end{proof}


\section{Remaining case}

\label{last}

P. Gruber proved that most convex surfaces have no simple closed geodesics
\cite{grub1}, and his proof can be easily adapted for most surfaces in
$\mathcal{A}(0,2)$. An important step in his proof was to find a dense set of
convex polyhedra without simple closed geodesics; this followed immediately
from the Gauss-Bonnet formula, because the curvature of a convex polyhedron is
concentrated at its vertices. This proof idea cannot be translated to
polyhedra in $\mathcal{A}(1,2)$, because, in our case, the curvature measure
is no longer supported by vertices.

\begin{lm}
\label{prepb}For any $a<2\pi$, any $1$-polyhedron $P\in\mathcal{A}\left(
1,2\right)  $ has at most finitely many closed almost-geodesics of length less
than $a$.
\end{lm}

\begin{proof}
A simple closed almost-geodesic which does not pass through any vertex is a
simple closed geodesic. Two such geodesics are necessarily intersecting, for
otherwise the topological cylinder they would bound would have to be flat by
the Gauss-Bonnet formula.

Assume there are infinitely many simple closed geodesics of length less than
$a$; by compactness (see Lemma \ref{LA}), one can find a sequence $G_{n}$
(with $\mathcal{\ell}\left(  G_{n}\right)  \leq a$) of distinct simple closed
geodesics converging to an almost-geodesic $G$. For $n$, $m$ large enough,
$G_{n}$ and $G_{m}$ are not separated by vertices. Hence each portion of
$G_{n}$ between two points of $G_{n}\cap G_{m}$ measures $\pi$. It follows
that $\mathcal{\ell}\left(  G_{n}\right)  \geq2\pi>a$.

Now choose a vertex $v$ and examine the simple closed almost-geodesics of
length at most $a$ passing through $v$. As precedently, if there are
infinitely many, one can find a sequence $G_{n}$ of such curves converging to
$G$. Obviously $v$ also belongs to $G$. For $n$, $m$ large enough $G_{n}$ and
$G_{m}$ are not separated by vertices, thus, if $G_{n}\cap G_{m}$ contains a
second point, then the previous argument applies and $\mathcal{\ell}\left(
G_{n}\right)  >a$.

Therefore, we can assume moreover that all curves $G_{n}$ lie in the same
half-surface bounded by $G$. Hence one can extract from $G_{n}$ a subsequence
such that $G_{m}$ lies between $G_{n}$ and $G$ for any $m>n$. Let $\alpha_{n}$
be the angle at point $v$ of the half surface bounded by $G_{n}$ and
containing $G$; the sequence $\alpha_{n}$ is decreasing, in contradiction with
the fact that all $G_{n}$ are supposed to be almost-geodesics.
\end{proof}

\begin{rmk}
We obtained a few properties of polyhedra in $\mathcal{A}\left(  1,2\right)
$, see Lemmas \ref{LLSCG}, \ref{LDPi}, \ref{LHS} and \ref{prepb}. Notice that
our polyhedra are different from the \emph{ball-polyhedra}, defined and
studied in a series of papers by K. Bezdek and his collaborators, see e.g.
\cite{Bezdek}.
\end{rmk}

\begin{lm}
\label{LDLC}For any $a<2\pi$, any surface $A\in\mathcal{A}\left(  1,2\right)
$ can be approximated by surfaces without simple closed geodesics of length at
most $a$.
\end{lm}

\begin{proof}
First approximate $A$ by a $1$-polyhedron $P\in\mathcal{A}\left(  1,2\right)
$. By Lemma \ref{prepb}, $P$ carries finitely many simple closed
almost-geodesics of length at most $a$. On this polyhedron, choose on each
simple closed geodesic $G$ of length at most $a$ a point $x_{G}$ which does
not belong to any other simple closed almost-geodesic of length at most $a$.

Consider the surface $P_{\varepsilon}$ obtained from $P$ in the following way.
First divide all distances on $P$ by $1+\varepsilon$, to obtain a $\left(
1+\varepsilon\right)  ^{2}$-polyhedron. Then, for each chosen point $x_{G}$,
cut out a small isosceles triangle $x_{G}y_{G}y_{G}^{\prime}$, symmetric with
respect to the geodesic normal to $G$ at $x_{G}$, such that $d\left(
x_{G},y_{G}\right)  =d\left(  x_{G},y_{G}^{\prime}\right)  =\varepsilon$ and
$\measuredangle y_{G}x_{G}y_{G}^{\prime}=\frac{\pi}{2}$. Then, replace this
triangle by a triangle $T_{G}$ of ${\mathbb{M}}_{1}$ with the same edge lengths.

In the rest of the proof we show that, for $\varepsilon$ small enough,
$P_{\varepsilon}$ has no simple closed geodesic of length at most $a$. Suppose
on the contrary that there exists a simple closed geodesic $G_{\varepsilon
}\subset P_{\varepsilon}$ such that $\mathcal{\ell}\left(  G_{\varepsilon
}\right)  \leq a$. Since the points $x_{G}$ are (corresponding to) vertices of
$P_{\varepsilon}$, $G_{\varepsilon}$ is not (corresponding to) a simple closed
geodesic of $P$. Hence $G_{\varepsilon}$ should pass across at least one
triangle $T_{G}$.

Denote by $G_{\varepsilon}^-$ the part of $G_{\varepsilon}$ outside the interior of all triangles $T_{G}$.
By compactness, $G_{\varepsilon}^-$ admits (at least) a limit curve
$G_{0}\subset P$, when $\varepsilon$ tends to $0$.
Since $G_{\varepsilon}^-$ can be seen as a curve on $P$, Lemma \ref{LLSCG} implies that
$G_{0}$ is an almost-geodesic through $x_{G}$, hence $G_{0}=G$. It follows
that, for small $\varepsilon$, $G_{\varepsilon}$ is included in a
neighbourhood $V_{G,\varepsilon}$ of $G$ in $P_{\varepsilon}$. Moreover, for
distinct simple closed geodesics $F$ and $G$, $V_{G,\varepsilon}\cap
x_{F}y_{F}y^{\prime}_{F}=\emptyset$.

Let $x_{G}^{\prime}$ be the point on $G$ which, together with $x_{G}$, divides
$G$ into two equally-long arcs. Denote by $N$ (resp. $N^{\prime}$) a geodesic
arc normal to $G$ through $x_{G}$ (resp. through $x_{G}^{\prime}$). Notice
that $V_{G,\varepsilon}$ may be chosen to be symmetrical with respect to $N$
(or, equivalently, with respect to $N^{\prime}$); denote by $s$ this symmetry;
we have $G=s(G)$, $N=s\left(  N\right)  $, $N^{\prime}=s(N^{\prime})$.

Assume first that $G_{\varepsilon}\neq s\left(  G_{\varepsilon}\right)  $.
Since $G_{\varepsilon}\cap\left(  N\cup N^{\prime}\right)  \subset
G_{\varepsilon}\cap s\left(  G_{\varepsilon}\right)  $, $G_{\varepsilon}$ and
$s\left(  G_{\varepsilon}\right)  $ intersect in at least two points, and so
define at least two digons, symmetric to each other and of perimeter
$2\Lambda_{\varepsilon}$. Now replace back $T_{\varepsilon}$ by a triangle of
curvature $(1+\varepsilon)^{2}$ and extend the remaining parts of
$G_{\varepsilon}$ and $s\left(  G_{\varepsilon}\right)  $ to complete the
digons. This produces two spherical digons of perimeter $2\pi/(1+\varepsilon
)$, and thus contradicts the fact that $\lim\Lambda_{\varepsilon}\leq a$.

Therefore, we may assume that $G_{\varepsilon}=s\left(  G_{\varepsilon
}\right)  $. We claim that $G_{\varepsilon}\cap G\neq\emptyset$. Suppose on
the contrary that $G_{\varepsilon}$ and $G$ are not intersecting. Then the
boundary of the topological cylinder $C$ between them has only one angle (at
$x_{G}$), of measure $\pi-\eta$, with $\eta>0$. By the Gauss-Bonnet formula,
the total curvature of $C$ should equal $-\eta$, which is obviously
impossible, hence $G_{\varepsilon}\cap G\neq\emptyset$.

Notice that $G_{\varepsilon}\cap G\neq\emptyset$ contains precisely two
points, because otherwise $G$ and $G_{\varepsilon}$ would determine at least
three digons, two of which would have perimeter $2\pi/(1+\varepsilon)$, and so
the length $G$ would be at least $2\pi/(1+\varepsilon)$, and its limit when
$\varepsilon$ goes to $0$ would be greater than $a$.

The next argument is illustrated by Figure \ref{Fig3}. Put $G_{\varepsilon
}\cap G=\{v_{G},v_{G}^{\prime}\}$ (with $v_{G}^{\prime}=s\left(  v_{G}\right)
$). $G$ and $G_{\varepsilon}$ are delimitating two digons, one of which is
spherical (because it doesn't intersect $T_{G}$) and has perimeter
$2\pi/\left(  1+\varepsilon\right)  $.

The geodesic $G_{\varepsilon}$ intersects the segments $x_{G}y_{G}$ and
$x_{G}y_{G}^{\prime}$ at $z_{G}$ and $z_{G}^{\prime}$ respectively. Let $\phi$
be the angle at $z_{G}$ of the geodesic triangle $x_{G}y_{G}z_{G}$.
\begin{figure}[ptb]
\begin{center}
\includegraphics[scale=1]{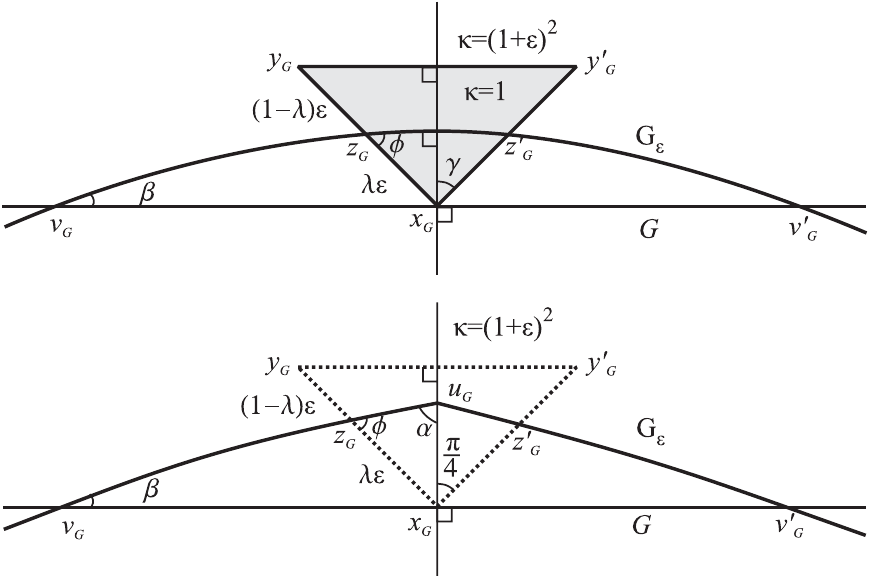}
\end{center}
\caption{Proof of Lemma \ref{LDLC}.}%
\label{Fig3}%
\end{figure}

Now cut out $T_{G}$ and glue back a triangle of curvature $\left(
1+\varepsilon\right)  ^{2}$; extend $G_{\varepsilon}$ beyond $z_{G}$ and
$z_{G}^{\prime}$ until it self-intersects, say at $u_{G}$. Denote by $2\alpha$
the angle of the quadrilateral $x_{G}z_{G}u_{G}z_{G}^{\prime}$ at $u_{G}$. Put
$\rho=1+\varepsilon$, $\lambda=d\left(  x_{G},z_{G}\right)  /\varepsilon$.

The rest of the proof consists in computing (a Taylor expansion of) $d\left(
v_{G},x_{G}\right)  $ as a function of $\varepsilon$, by means of spherical trigonometry.

Denote by $2\gamma$ the angle of $T_{G}$ at point $x_{G}$. Using twice the law
of sines, one can compute
\begin{align*}
\gamma & =\arcsin\frac{\sin\left(  \frac{1}{\rho}\arcsin\left(  \sin\frac{\pi
}{4}\sin\rho\varepsilon\right)  \right)  }{\sin\varepsilon}\\
& =\frac{\pi}{4}-\frac{1}{6}\varepsilon^{3}+O\left(  \varepsilon^{4}\right)
\text{.}%
\end{align*}

The law of cosines for angles in one half of the triangle $x_{G}z_{G}%
z_{G}^{\prime}\subset T_{G}$ gives
\[
\cos\frac{\pi}{2}=-\cos\phi\cos\gamma+\sin\phi\sin\gamma\cos\lambda
\varepsilon\text{,}%
\]
whence
\begin{align*}
\tan\phi &  =\frac{1}{\tan\gamma\cos\lambda\varepsilon}\\
&  =1+\frac{\lambda^{2}\varepsilon^{2}}{2}+\frac{\varepsilon^{3}}{3}+O\left(
\varepsilon^{4}\right)  \text{.}%
\end{align*}
By straightforward computations
\begin{align*}
\sin\phi &  =\frac{\sqrt{2}}{2}\left(  1+\frac{\lambda^{2}\varepsilon^{2}}%
{4}+\frac{\varepsilon^{3}}{6}\right)  +O\left(  \varepsilon^{4}\right)
\text{,}\\
\cos\phi &  =\frac{\sqrt{2}}{2}\left(  1-\frac{\lambda^{2}\varepsilon^{2}}%
{4}-\frac{\varepsilon^{3}}{6}\right)  +O\left(  \varepsilon^{4}\right)
\text{.}%
\end{align*}
The law of cosines for angles in the triangle $u_{G}x_{G}z_{G}$ gives
\begin{align*}
\cos\alpha &  =-\cos\phi\cos\frac{\pi}{4}+\sin\phi\sin\frac{\pi}{4}\cos
\rho\lambda\varepsilon\\
&  =\left(  \frac{1}{6}-\frac{1}{2}\lambda^{2}\right)  \varepsilon
^{3}+O\left(  \varepsilon^{4}\right)  \text{.}%
\end{align*}
The law of cosines for angles in the triangle $v_{G}x_{G}z_{G}$ gives
\begin{align*}
\cos\beta &  =-\cos\left(  \pi-\phi\right)  \cos\frac{\pi}{4}+\sin\left(
\pi-\phi\right)  \sin\frac{\pi}{4}\cos\rho\lambda\varepsilon\\
&  =1-\lambda^{2}\left(  \frac{\varepsilon^{2}}{4}+\frac{\varepsilon^{3}}%
{2}\right)  +O\left(  \varepsilon^{4}\right)  \text{,}%
\end{align*}
whence
\[
\sin\beta=\frac{\sqrt{2}}{2}\lambda\varepsilon\left(  1+\varepsilon\right)
+O\left(  \varepsilon^{3}\right)  \text{.}%
\]
At last, the law of sines in the same triangle $v_{G}x_{G}z_{G}$ yields
\begin{align*}
\sin\rho d\left(  v_{G},x_{G}\right)   &  =\frac{\sin\rho\lambda\varepsilon
}{\sin\beta}\sin\phi\\
&  =1+O\left(  \varepsilon\right)  \text{.}%
\end{align*}
On the other hand, $d\left(  v_{G}x_{G}\right)  $ does not depend on
$\varepsilon$, and so is equal to $\pi/2$. Hence the length of $G$ is $2\pi$
and we get a contradiction. This ends the proof.
\end{proof}

\begin{thm}
\label{A(1,2)} Most $A\in\mathcal{A}\left(  1,2\right)  $ have no simple
closed geodesic.
\end{thm}

\begin{proof}
A closed geodesic on $A\in\mathcal{A}\left(  1,2\right)  $ is seen as a map
from $\mathbb{R}/\mathbb{Z}$ to $A$; its parameter is assumed proportional to
the arc-length. For a given surface $A$, define $\mathcal{H}_{A}\left(
\varepsilon,\eta,a\right)  $ as the set of all simple closed geodesics $G$ of
$A$ such that \textit{(i)} for any $t\in\mathbb{R}/\mathbb{Z}$ and any
$s\in\left[  0,\varepsilon\right]  $, $d\left(  \gamma\left(  t\right)
,\gamma\left(  t+s\right)  \right)  =s\mathcal{\ell}\left(  G\right)  $,
\textit{(ii)} for any points $x,y\in G$ whose distance along $G$ is at least
$\varepsilon$, we have $d_{A}(x,y)\geq\eta$, and \textit{(iii)} $\mathcal{\ell
}\left(  G\right)  \leq a$.

Denote by $\mathcal{M}_{pqr}$ the set of all $A\in\mathcal{A}\left(
1,2\right)  $ such that $\mathcal{H}_{A}\left(  \frac{1}{p},\frac{1}{q}%
,2\pi-\frac{1}{r}\right)  $ is nonempty.

We have to prove that the set
\[
\mathcal{M}\overset{\mathrm{def}}{=}\set(:A\in\mathcal{A}\left(  1,2\right)
|A~\text{has a simple closed geodesic}:)
\]
is meager. By Lemma \ref{length}, we have
\[
\mathcal{M}=\left\{  S_{0}\right\}  \cup\bigcup_{p,q,r\in\mathbb{N}^{\ast}%
}\mathcal{M}_{pqr}\text{.}%
\]

Each set $\mathcal{M}_{pqr}$ has empty interior by Lemma \ref{LDLC}; we show
next that it is closed. Let $A_{n}\in\mathcal{M}_{pqr}$ be a sequence
converging to $A\in\mathcal{A}\left(  1,2\right)  $. By Lemma \ref{SEL}, we
can assume that $A_{n}$ and $A$ are embedded in the same compact metric space
$Z$. Let $G_{n}$ be a geodesic in $\mathcal{H}_{A_{n}}\left(  \frac{1}%
{p},\frac{1}{q},2\pi-\frac{1}{r}\right)  $. Notice that $\mathcal{\ell}\left(
G_{n}\right)  <2\pi$, hence by Ascoli's theorem we can extract from $G_{n}$ a
converging subsequence; denote by $G:\mathbb{R}/\mathbb{Z}\rightarrow A$ its
limit. Since $\mathcal{\ell}$ is lower semi-continuous, $G$ belongs to
$\mathcal{H}_{A}\left(  \frac{1}{p},\frac{1}{q},2\pi-\frac{1}{r}\right)  $.
This ends the proof.
\end{proof}


\section{Conclusions}

\label{conclusions}

Gathering together Theorems \ref{TNC}, \ref{generic_scg} and \ref{A(1,2)}, we get

\begin{sumthm}
i) For $\kappa=1$ we have: \newline i.1) most surfaces in $\mathcal{A}(1,1)$
have infinitely many simple closed geodesics; \newline i.2) most surfaces in
$\mathcal{A}(1,2)$ have no simple closed geodesic.

ii) For $\kappa=0$ we have: \newline ii.1) most surfaces in $\mathcal{A}(0,2)$
have no closed geodesic; \newline ii.2) most surfaces in $\mathcal{A}(0,1)$
have infinitely many simple closed geodesics; \newline ii.3) all surfaces in
$\mathcal{A}(0,0)$ are unions of simple closed geodesics.

iii) Most surfaces in $\mathcal{A}(-1)$ have infinitely many non-intersecting
simple closed geodesics.
\end{sumthm}

\begin{rmk}
P. Gruber proved that most convex surfaces have no closed geodesics
\cite{grub}, and his proof yields the above result on most surfaces in
$\mathcal{A}(0,2)$. Whether most surfaces in $\mathcal{A}(1,2)$ do not have
non-simple closed geodesics remains an open question.

It is also an open question whether a typical surface in $\mathcal{A}(-1)$ or
in $\mathcal{A}(\kappa,1)$ also has infinitely many non-simple closed
geodesics of a given \textquotedblleft flat knot type\textquotedblright\ (with
the terminology in \cite{Angenent}).
\end{rmk}

Our final remark concerns the length spectrum of Alexandrov surfaces.

\begin{rmk}
One can also consider lengths in the statements of Theorems \ref{TNC} and
\ref{generic_scg}. Put $\mathcal{B}(-1)= \mathcal{A}(-1)$, $\mathcal{B}(0)=
\mathcal{A}(0,1)$ and $\mathcal{B}(1)= \mathcal{A}(1,1)$. With the very same
proof ideas, but varying the parameters $\lambda$ and $\varepsilon$, one can
prove the following statement.

Let $\kappa\in\{-1,0,1\}$; for any $\delta>0$ there exists a residual set
$\mathcal{C}$ in $\mathcal{B}(\kappa)$ such that, for any $A\in\mathcal{C}$,
there exist $L>0$ and infinitely many simple closed geodesics on $A$ whose
lengths are pairwise different and belong to $[L,L+\delta]$.
\end{rmk}


\bigskip

\noindent\textbf{Acknowledgement.} The authors were partly supported by the
grant PN-II-ID-PCE-2011-3-0533 of the Romanian National Authority for
Scientific Research, CNCS-UEFISCDI.

They also express thanks to Tudor Zamfirecu for suggesting them to investigate
properties of most Alexandrov surfaces.


{\small \bigskip}

{\small J\"oel Rouyer }

{\small \noindent Institute of Mathematics ``Simion Stoilow'' of the Romanian
Academy, \newline P.O. Box 1-764, Bucharest 70700, ROMANIA \newline
Joel.Rouyer@ymail.com, Joel.Rouyer@imar.ro }

{\small \medskip}

{\small Costin V\^{\i}lcu }

{\small \noindent Institute of Mathematics ``Simion Stoilow'' of the Romanian
Academy, \newline P.O. Box 1-764, Bucharest 70700, ROMANIA \newline
Costin.Vilcu@imar.ro }


\begin{thebibliography}{99}                                                                                               %


\bibitem {A-Z}K. Adiprasito and T. Zamfirescu, \textit{Few Alexandrov spaces
are Riemannian}, submitted, 2012

\bibitem {al}A. D. Alexandrov, \textsl{Die innere Geometrie der konvexen
Fl\"{a}chen}, Akademie-Verlag, Berlin, 1955

\bibitem {Angenent}S. Angenent, \textit{Curve Shortening and the topology of
closed geodesics on surfaces}, Ann. Math. \textbf{162} (2005), 1185--1239

\bibitem {ba}V. Bangert, \textit{On the existence of closed geodesics on
two-spheres}, Int. J. Math. \textbf{4} (1993), 1--10

\bibitem {Bezdek}K. Bezdek, \textit{Ball-polyhedra as intersections of
congruent balls}, in \textsl{Classical Topics in Discrete Geometry}, 57--68,
CMS Books in Mathematics, Springer New York, 2010

\bibitem {BBI}D. Burago, Yu. Burago, and S. Ivanov, \textsl{A course in metric
geometry}, American Mathematical Society, Providence, Rhode Island, 2001.

\bibitem {BGP}Yu. Burago, M.~Gromov, and G.~Perel'man, \textit{A. D.
Alexandrov spaces with curvature bounded below.}, Russ. Math. Surv.
\textbf{47} (1992), 1--58 (English. Russian original)

\bibitem {Contreras}G. Contreras, \textit{Geodesic flows with positive
topological entropy, twist maps and hyperbolicity}, Ann. Math. \textbf{172}
(2010), 761--808

\bibitem {f}J. Franks, \textit{Geodesics on $S\sp 2$ and periodic points of
annulus homeomorphisms}, Invent. Math. \textbf{108} (1992), 403--418

\bibitem {GPL}M. Gromov, J. Lafontaine, and P. Pansu, \textsl{Structure
m\'etrique pour les vari\'et\'es riemanniennes}, CEDIC/Fenand Nathan, 1981

\bibitem {grub1}P. Gruber, \textit{Geodesics on typical convex surfaces}, Atti
Accad. Naz. Lincei Rend. Cl. Sci. Fis. Mat. Natur. \textbf{82} (1988), 651--659

\bibitem {grub}P. Gruber, \textit{A typical convex surface contains no closed
geodesic}, J. Reine Angew. Math. \textbf{416} (1991), 195--205

\bibitem {gw}P. Gruber, \textit{Baire categories in convexity}, in P. Gruber
and J. Wills (eds.), \textit{Handbook of Convex Geometry}, vol. B,
North-Holland, Amsterdam, 1993, 1327--1346

\bibitem {Hadamard}J. Hadamard, \textit{Les surfaces \`a courbures oppos\'ees
et leurs lignes g\'eod\'esique}, J. Math. Pure Appl. \textbf{4} (1898), 27--75

\bibitem {IRV2}J. Itoh, J. Rouyer, and C. V{\^{\i}}lcu, \textit{Moderate
smoothness of most Alexandrov surfaces}, arXiv:1308.3862 [math.MG]

\bibitem {Kapo1}V. Kapovitch, \textit{Perelman's stability theorem}, J.
Cheeger et al. (eds.), \textsl{Metric and comparison geometry}. International
Press. Surveys in Differential Geometry 11 (2007), 103--136

\bibitem {k}W. Klingenberg, \textsl{Riemannian Geometry}, De Gruyter,
Berlin-New York, 1982

\bibitem {Mirzakhani}M. Mirzakhani, \textit{Growth of the number of simple
closed geodesics on hyperbolic surfaces}, Ann. Math. \textbf{168} (2008), 97--125

\bibitem {Per1}G. Perel'man, \emph{A. D. Alexandrov spaces with curvatures
bounded from below {II}}, preprint 1991

\bibitem {p}A. V. Pogorelov, \textit{Quasigeodesic lines on convex surfaces},
Mat. Sb. \textbf{25} (1949), 275--307

\bibitem {Pog}A. V. Pogorelov, \textsl{Extrinsic geometry of convex surfaces},
Amer. Math. Soc., 1973

\bibitem {Rademacher}H. Rademacher, \textit{On a generic property of geodesic
flows}, Math. Ann. \textbf{298} (1994), 101--116

\bibitem {JR7}J. Rouyer, \textit{Generic properties of compact metric spaces},
Topology Appl. \textbf{158} (2011), 2140--2147

\bibitem {RV2}J. Rouyer and C. V{\^{\i}}lcu, \textit{The connected components
of the space of Alexandrov surfaces}, arXiv:1310.8491 [math.MG]

\bibitem {Shiohama92}K. Shiohama, \textsl{An introduction to the geometry of
Alexandrov spaces}, Lecture Notes Series, Seoul National University, 1992

\bibitem {t}V. A. Toponogov, \textit{Computation of the length of a closed
geodesic on a convex surface}, Dokl. Akad. Nauk SSSR \textbf{124} (1959),
282-284 (in Russian)

\bibitem {Troyanov}M. Troyanov, \textit{Metrics of constant curvature on a
sphere with two conical singularities}, \textsl{Third International Symposium
on Differential Geometry at Pe\~niscola 1988}, Pe\~niscola, 296--306, Lecture
Notes in Math. 1410, Springer, Berlin, 1989

\bibitem {z-b}T. Zamfirescu, \textit{Baire categories in convexity}, Atti Sem.
Mat. Fis. Univ. Modena \textbf{39} (1991), 139--164
\end{thebibliography}
\end{document}